\documentclass{amsart}
\usepackage{graphicx} % Required for inserting images
\usepackage{amsmath,amsfonts}
\usepackage{color}
\usepackage{tcolorbox}
\usepackage{amssymb}
\usepackage{cite}
\usepackage{enumitem}
\usepackage{gensymb}
\usepackage{float}
\usepackage{amsthm}
\usepackage{siunitx}
\usepackage[mathscr]{euscript}
\newtheorem{theorem}{Theorem}[section]

\newtheorem{lemma}[theorem]{Lemma}
\newtheorem{proposition}[theorem]{Proposition}

\newtheorem{remark}[theorem]{Remark}

\usepackage{tikz}
\usepackage{pgfplots}
\usepackage{caption}
\usepackage{subcaption}
\usepackage{amsrefs}
\usepackage{tcolorbox}
\usepackage{MnSymbol}
\newcommand{\R}{\mathbb{R}}
\newcommand{\N}{\mathbb{N}}

\newcommand{\p}{\partial}
\newcommand{\eps}{\varepsilon}

\newcommand{\dist}{\textnormal{dist}}
\newcommand{\diam}{\textnormal{diam}}

\newcommand{\norm}[1]{\left\lVert#1\right\rVert}
\newcommand{\loc}{\textnormal{loc}}

\renewcommand{\div}{\textnormal{div}}
\newcommand{\abs}[1]{\left\lvert#1\right\rvert}
\numberwithin{equation}{section}
\title[A singular perturbation approach]{A singular perturbation approach to the Dirichlet-area minimisation problem}
\author{Anthony Salib}
\address{Department of Mathematics, University of Duisburg-Essen, Thea-Leymann-Strasse 9, 45127 Essen,Germany.}
\email{anthony.salib@stud.uni-due.de}
\author{Georg S. Weiss}
\address{Department of Mathematics, University of Duisburg-Essen, Thea-Leymann-Strasse 9, 45127 Essen,Germany.}
\email{georg.weiss@uni-due.de}

\begin{document}
\begin{abstract}
	We study both one and two-phase minimisers of the Dirichlet-area energy 
	\begin{equation*}
		E(v) = \int_{B_1} \abs{\nabla v}^2 + Per(\{v>0\},B_1).
	\end{equation*}
	In the two-phase case, we show that the energies
	\begin{equation*}
	 E_{\eps}(v) = \int_{B_1}\abs{\nabla v}^2 + \frac{1}{\eps}W\left(\frac{v}{\eps^{1/2}}\right),
	\end{equation*}
	$\Gamma$-converge to $E$ as $\eps \to 0$, where $W$ is the double well potential extended by zero outside of $[-1,1]$ . As a consequence, we show that bounded local minimisers  of $E_{\eps}$ converge to a local minimiser of $E$.
\end{abstract}
\maketitle
\begin{center}
	\textit{In honor of Nina Uraltseva for her 90th birthday.}
\end{center}
\tableofcontents
\section{Introduction}
We study minimisers of the Dirichlet-area energy
\begin{equation}
	\label{eqn:dirichlet:area}
	 E(v)= \int_{B_1}\abs{\nabla v}^2 + Per(\{v>0\},B_1),
\end{equation}
where the perimeter functional is given by
	\begin{equation*}
		Per(\Omega,B_1) = \sup_{\xi \in C^{1}_0(B_1;\R^n), \norm{\xi}_{L^{\infty}} \leq 1} \int_{\Omega} \div (\xi).
	\end{equation*}
This energy was first introduced in \cite{athanasopoulos2001} and there the authors gave a complete study of two-phase minimisers and their free boundaries. It is also stated (without proof) in the introduction of \cite{athanasopoulos2001} that minimisers of \eqref{eqn:dirichlet:area} should appear as limits of minimisers of the energies
\begin{equation*}
	\label{eqn:eps:energy}
	E_{\eps}(v)=\int_{B_1} \left(\abs{\nabla v}^2 + \frac{1}{\eps}W\left(\frac{v}{\eps^{\frac{1}{2}}}\right)\right)
\end{equation*}
as $\eps \to 0$, where 
\begin{equation*}
	W(t) =
	\begin{cases}
		(1-t^2)^2 & \text{for } \abs{t} \leq 1,\\
		0 & \text{for } \abs{t} >1.
	\end{cases}
\end{equation*}

On the other hand, the one-phase problem was recently studied in \cite{savinAltPhillips} where it was shown that suitable rescalings of the functionals
\begin{equation*}
	\int_{B_1}\left(\abs{\nabla v}^2 + v^{\gamma}\chi_{\{v>0\}}\right)
\end{equation*}
$\Gamma$-converge to \eqref{eqn:dirichlet:area} as $\gamma \to -2$. We remark here that there exist several results in the literature on one-phase minimisers under additional constraints. For instance, minimising \eqref{eqn:dirichlet:area} under a \textit{volume constraint} of the zero level set, one phase minimisers have been studied in \cite{jiang2007} (see also \cite{lin2004New,silvestre2011full} for the two phase analogue with volume constraint). Moreover, under zero boundary conditions and an interior positivity condition, one-phase minimisers of
\begin{equation*}
	\int_{B_1} F(\abs{\nabla v}) + Per(\{v>0\},B_1),
\end{equation*}
for suitable $F$ have been studied in \cite{mazzone2003,haykshah2013convex,wolandski2021pLapPer}. As we will show in Proposition \ref{proposition:nonexistence}, some additional constraint is necessary in order to ensure the existence of a free boundary in the one-phase setting. Although simple, we have not found a proof of this non-existence result in the literature, and so we include it here. 

The aim of this current note is to rigorously establish the singular perturbation approach claimed in \cite{athanasopoulos2001} for the minimisation problem \eqref{eqn:dirichlet:area}. Since one-phase minimisers in our setting must be strictly positive in the interior of $B_1$, we will establish the $\Gamma$-convergence for true two-phase solutions. Using the $\Gamma$-convergence result we then show that local minimisers of $E_{\eps}(\cdot,B_1)$  which satisfy a uniform $L^{\infty}$ bound converge to a local minimiser of $E_{B_1}$.

\section{Preliminaries and main results}\label{section:preliminaries}
We first show that one-phase minimisers of \eqref{eqn:dirichlet:area} do not exhibit a free boundary. 
\begin{proposition}\label{proposition:nonexistence}
	Suppose that $u \in H^1(B_1)$ (not identically $0$) is a non-negative minimiser of $E$ in $B_1$, i.e. $E(u)\leq E(v)$ for all $v$ such that $u-v \in H^1_0(B_1)$. Then $u$ is strictly positive in the interior of $B_1$.
\end{proposition}
\begin{proof}
	Firstly if $u$ is harmonic then by the strong maximum principle we must have that $u$ is positive in the interior of $B_1$. We now suppose that $u$ is not harmonic in $B_1$. Since $u$ is not identically zero it has non-zero trace on $\p B_1$ and so for the function $v \in H^1(B_1)$ satisfying $v = u$ on $\p B_1$ and $\Delta v =0$ in $B_1$, 
	there holds that $u-v \in H^1_0(B_1)$ and
	\begin{equation}
		\label{equation:compared:dirichlet}
		\int_{B_1} \abs{\nabla v}^2 - \int_{B_1} \abs{\nabla u}^2 < 0.
	\end{equation}
	 Moreover, since $v$ is harmonic we have that $v >0$ in $B_1$ and so $Per(\{v>0\}\cap B_1)=0$. This coupled with \eqref{equation:compared:dirichlet} yields
	\begin{equation*}
		E(v) < E(u),
	\end{equation*}
	contradicting the minimality of $u$.
\end{proof}

We will now describe our results regarding two-phase minimisers. Given any bounded domain $A\subset \R^n$ with Lipschitz boundary, we will consider the minimisation of the energies
\begin{equation}
	\label{eqn:eps:energydef}
	E_{\eps}(v, A) =  \int_{A} \left(\abs{\nabla v}^2 + \frac{1}{\eps}W\left(\frac{v}{\eps^{\frac{1}{2}}}\right)\right),
\end{equation}
over the set
\begin{equation*}
	\tilde{\mathscr{A}}(A) = \left\{v: v \in H^1(A)\right\}.
\end{equation*}

We define the increasing function
\begin{equation*}
	\label{equation:bigH:definition}
	H(t) = \int_{0}^t 2\sqrt{W(s)}ds,
\end{equation*}
as well as the constant $c_0 = 2H(1)$. We will also have need to define the renormalised function 
\begin{equation*}
	\tilde{H}(t) =\frac{2}{c_0}H(t)
\end{equation*}
satisfying $\tilde{H}(t) = 1$ for $t \geq 1$ and $\tilde{H}(t) = -1$ for $t \leq -1$. 
\begin{remark}
	If $u_{\eps_k}$ are local minimisers of $E_{\eps_k}$, then the functions $\tilde{H}\left(\frac{u_{\eps_k}}{\eps_k^{1/2}}\right)$ are compact in $L^1$, see Theorem \ref{theorem:main:existence} below. 
\end{remark}

Finally, we will consider the energy
\begin{equation}
	\label{eqn:energy:tech}
	E_{A}(v,\Omega) = \int_{A}\abs{\nabla v}^2 + c_0Per(\Omega,A),
\end{equation}
where the minimisation occurs over the set 
\begin{align*}
	\mathscr{A}(A) = \{(v,\Omega): & v \in H^1(A),\\ &\text{ $\Omega$ a set finite perimeter such that } v\vert_{\Omega\cap A} \geq 0 \text{ and } v\vert_{\Omega^c\cap A} \leq 0 \text{ a.e. }\}.
\end{align*}

Our first result is that the sequence $E_{\eps}(\cdot,A)$ defined in \eqref{eqn:eps:energydef} $\Gamma$-converge to the energy $E_A$ defined in \eqref{eqn:energy:tech}. 
\begin{theorem}[Gamma convergence]
	\label{theorem:gamma:convergence}
	Let $A \subset \R^n$ be bounded with Lipschitz boundary. Then the functionals $E_{\eps}(\cdot, A)$ $\Gamma$-converge to $E_A$ as $\eps \to 0$. That is, we have the following:
	\begin{enumerate}
		\item if $\eps_k \to 0$ and there holds that $u_{\eps_k} \to u$ in $L^2(A)$ and $\tilde{H}\left(\frac{u_{\eps_k}}{\eps_k^{1/2}}\right) \to \chi_{\Omega}-\chi_{\Omega^c}$ in $L^1(A)$ as $k \to \infty$, then 
		\begin{equation*}
				\liminf_{k\to\infty} E_{\eps_k}(u_{\eps_k},A) \geq E_A(u,\Omega)
		\end{equation*}
		\item if $(u,\Omega) \in \mathscr{A}(A)$ there exists $\eps_k \to 0$ and a sequence $\{u_{\eps_k}\}_{k\in \N}\subset H^1(A)$ such that $u_{\eps_k} \to u$ in $L^2(A)$ and $\tilde{H}\left(\frac{u_{\eps_k}}{\eps_k^{1/2}}\right) \to \chi_{\Omega}-\chi_{\Omega^c}$ in $L^1(A)$ as $k \to \infty$, and 
		\begin{equation*}
			\limsup_{k\to\infty} E_{\eps_k}(u_{\eps_k},B_1) \leq E_A(u,\Omega)
		\end{equation*}
	\end{enumerate}
\end{theorem}

For any $M > 0$ we define the sets 
\begin{equation*}
	\tilde{\mathscr{A}}_M(A) = \left\{v \in H^1(A): \abs{v} \leq M\right\},
\end{equation*}
and 
\begin{align*}
	\mathscr{A}_M(A) = \{(v,\Omega): & v \in H^1(A), \abs{v}\leq M,\\ &\text{ $\Omega$ a set finite perimeter such that } v\vert_{\Omega\cap A} \geq 0 \text{ and } v\vert_{\Omega^c\cap A} \leq 0 \text{ a.e. }\}.
\end{align*}
We now have the following result concerning bounded local minimisers. 
\begin{theorem}
	\label{theorem:main:existence}
	Let $\eps_k \to 0$ as $k \to \infty$ and suppose that $u_{\eps_k} \in \tilde{\mathscr{A}}_M$ satisfies 
	\begin{equation*}
		E_{\eps_k}(u_{\eps_k},B_1) \leq E_{\eps_k}(v,B_1)
	\end{equation*}
	for each $v \in \tilde{\mathscr{A}}_M$ satisfying $u_{\eps_k} - v \in H^1_0(B_1)$. 
	Then, up to a subsequence, we have that
	\begin{equation*}
		u_{\eps_k} \to u \text{ in } L^2_{\loc}(B_1),
	\end{equation*}
	and
	\begin{equation*}
	 \tilde{H}\left(\frac{u_{\eps_k}}{\eps_k^{1/2}}\right) \to \chi_{\Omega}-\chi_{\Omega^c} \text{ in } L^1_{\loc}(B_1),
	\end{equation*}
	where $(u,\Omega) \in \mathscr{A}_M$ is a local minimiser of the energy $E_{B_1}$ over the set $\mathscr{A}_M$.
\end{theorem}

\section{$\Gamma$-convergence}\label{section:gammaconvergence}

We will first collect some properties of $1$D local minimisers $E_{\eps}$. Notice that if $\varphi_{\eps}$ is a $1$D local minimiser, it satisfies the Euler-Lagrange equation
\begin{equation}
	\label{equation:1d:EL}
	2\varphi''_{\eps} = \frac{1}{\eps^{3/2}}W'\left(\frac{\varphi_{\eps}}{\eps^{1/2}}\right),
\end{equation}
so that multiplying with $\varphi_{\eps}'$ we find, assuming that $\varphi_{\eps}' \geq 0$, that
\begin{equation}
	\label{equation:1d:family}
	(\varphi_{\eps}')^2 = \frac{1}{\eps}W\left(\frac{\varphi_{\eps}}{\eps^{1/2}}\right) + \left[(\varphi_{\eps}'(0))^2 -\frac{1}{\eps}W\left(\frac{\varphi_{\eps}(0)}{\eps^{1/2}}\right) \right].
\end{equation}
This property gives a family of 1D solutions depending on the initial conditions $\varphi_{\eps}'(0)$ and  $\varphi_{\eps}(0)$. 
\begin{lemma}
	\label{lemma:1d:properties}
	Let $\varphi_{\eps}$ solve \eqref{equation:1d:EL} with $\varphi_{\eps}(0) =0$ and $\varphi_{\eps}'(0)=\eps^{-1/2}$. Given $\kappa \in (0,1)$ there exists a $t_{\eps} = t_{\eps}(\kappa)$ such that $\abs{\varphi_{\eps}(s)}\geq \eps^{1/2}(1-\kappa)$ for $\abs{s}\geq t_{\eps}$ and $t_{\eps} \to 0$ as $\eps \to 0$. Moreover, for each $\gamma >0$ there exists a $\kappa >0$ such that for each $\eps > 0$ there holds that
	\begin{equation}
		\label{equation:energy:compared}
		\sup_{ t_{\eps} \leq \abs{s} \leq L } \frac{1}{\eps}W\left(\frac{\varphi_{\eps}(s)}{\eps^{1/2}}\right) \leq \gamma,
	\end{equation}
	for any $L > t_{\eps}$. Finally, $\varphi_{\eps} \to 0$ uniformly as $\eps \to 0$. 
\end{lemma}
\begin{proof}
	Since $\varphi_{\eps}(0) =0$ and $\varphi_{\eps}'(0)=\eps^{-1/2}$ we have from \eqref{equation:1d:family} that
	\begin{equation}
		\label{equation:1d:property}
		\varphi_{\eps}' = \frac{1}{\eps^{1/2}}\sqrt{W\left(\frac{\varphi_{\eps}}{\eps^{1/2}}\right)} .
	\end{equation}
	The behaviour of $\varphi_{\eps}$ and $t_{\eps}$ is then clear from \eqref{equation:1d:property} and \eqref{equation:1d:EL} since we can write $\varphi_{\eps} = \eps^{1/2} \varphi_1(\eps^{-1} \cdot)$. 
	To establish \eqref{equation:energy:compared}, observe that given  $\gamma > 0$ there exists a $\kappa > 0$ such that
	$$\abs{\varphi_{\eps}'(s) } \leq \gamma^{1/2},$$
	for any $\abs{s} \geq t_{\eps}$. Using \eqref{equation:1d:property} once more we have that 
	\begin{equation*}
		\frac{1}{\eps^{1/2}}\sqrt{W\left(\frac{\varphi_{\eps}}{\eps^{1/2}}\right)} \leq \gamma^{1/2},
	\end{equation*}
	which establishes \eqref{equation:energy:compared}.
\end{proof}

Lemma \ref{lemma:1d:properties} establishes the existence of a local minimiser $\varphi_{\eps}$ which satisfies $\varphi_{\eps}(s) \to \pm \eps^{1/2}$ as $s \to \pm \infty$. We will also have need for 1D solutions which are linear outside of an interval around $0$, which is the content of the following Lemma. 
\begin{lemma}
	\label{lemma:1D:properties:theta}
	Let $\varphi_{\theta,\eps}$ solve \eqref{equation:1d:EL} with $\varphi_{\theta,\eps}(0) =0$ and $\varphi_{\theta,\eps}'(0)=\sqrt{\eps^{-1}+\theta}$. Then there exists a $t_{\eps}$ such that $\abs{\varphi_{\theta,\eps}'(s)} =  \theta$ for $\abs{s} \geq t_{\eps}$, where $t_{\eps} \to 0$ as $\eps \to 0$. Finally, $\varphi_{\theta,\eps}(s) \to \theta s$ uniformly as $\eps \to 0$. 
\end{lemma}
\begin{proof}
	Thanks to the initial conditions we have that
	\begin{equation}
		\label{equation:1d:theta}
		\varphi_{\theta,\eps}'= \sqrt{\frac{1}{\eps}W\left(\frac{\varphi_{\theta,\eps}}{\eps^{1/2}}\right) + \theta}.
	\end{equation}
	Now noticing that the $\theta$ in \eqref{equation:1d:theta} ensures $\varphi_{\theta,\eps}$ is always increasing with gradient at least $\theta$, there must exist a $t_{\eps}$ such that $\varphi_{\theta,\eps}(\pm t_{\eps}) = \pm \eps^{1/2}$ and so $W\left(\frac{\varphi_{\theta,\eps}(s)}{\eps^{1/2}}\right) = 0$ for $\abs{s} \geq t_{\eps}$. The behaviour of $t_{\eps}$ is then deduced from the fact that $\varphi_{\theta,\eps} = \eps^{1/2}\varphi_{\theta\eps,1}(\eps^{-1}\cdot)$ as in the proof of Lemma \ref{lemma:1d:properties}. Since $t_{\eps} \to 0$ we have the uniform convergence of $\varphi_{\theta,\eps}(s)$ to $\theta s$.  
\end{proof}

We now give the proof of Theorem \ref{theorem:gamma:convergence} in the following two Lemmas. 

\begin{lemma}
	\label{lemma:gamma:liminf}
	Suppose that $\eps_k \to 0$ and that $u_{\eps_k}\to u$ in $L^2(A)$ and $\tilde{H}\left(\frac{u_{\eps_k}}{\eps_k^{1/2}}\right) \to \chi_{\Omega}-\chi_{\Omega^c}$ in $L^1(A)$ as $k \to \infty$ for some measurable set $\Omega$. Then 
	\begin{equation*}
		\label{equation:liminf:eq}
		\liminf_{k\to\infty} E_{\eps_k}(u_{\eps_k},A) \geq E_A(u,\Omega). 
	\end{equation*}
\end{lemma}
\begin{proof}
	We first split
	\begin{align*}
		E_{\eps_k}(u_{\eps_k},A)
		&=\int_{A}\left( \abs{\nabla u_{\eps_k}}^2 + \frac{1}{\eps_k}W\left(\frac{u_{\eps_k}}{\eps_k^{1/2}}\right)\right)\\
		&=\int_{A\cap\{-\eps_k^{1/2} < u_{\eps_k} < \eps_k^{1/2}\}} \left( \abs{\nabla u_{\eps_k}}^2 + \frac{1}{\eps_k}W\left(\frac{u_{\eps_k}}{\eps_k^{1/2}}\right)\right)
		+ \int_{A\cap\{\abs{u_k} \geq \eps_k^{1/2}\}}\left( \abs{\nabla u_{\eps_k}}^2 + \frac{1}{\eps_k}W\left(\frac{u_{\eps_k}}{\eps_k^{1/2}}\right)\right)\\
		&= \int_{A\cap\{-\eps_k^{1/2} < u_{\eps_k} < \eps_k^{1/2}\}} \left( \abs{\nabla u_{\eps_k}}^2 + \frac{1}{\eps_k}W\left(\frac{u_{\eps_k}}{\eps_k^{1/2}}\right)\right)
		+ \int_{A\cap\{\abs{u_k} \geq \eps_k^{1/2}\}}\abs{\nabla u_{\eps_k}}^2,
	\end{align*}
	since $W(s)=0$ for $\abs{s}\geq 1$. Now, by Young's inequality we have
	\begin{align*}
		\int_{A \cap \{-\eps_k^{1/2} < u_{\eps_k} <\eps_k^{1/2}\}} \left(\abs{\nabla u_{\eps_k}}^2 + \frac{1}{\eps_k}W\left(\frac{u_{\eps_k}}{\eps_k^{1/2}}\right)\right)
		&\geq \int_{A\cap\{-\eps_k^{1/2} < u_{\eps_k} < \eps_k^{1/2}\}} \left(2\abs{\nabla u_{\eps_k}}\frac{1}{\eps_k^{1/2}}\sqrt{W\left(\frac{u_{\eps_k}}{\eps_k^{1/2}}\right)}\right)\\
		&= \int_{A\cap\{\eps_k^{1/2} < u_{\eps_k} < \eps_k^{1/2}\}}\abs{\nabla H\left(\frac{u_{\eps_k}}{\eps_k^{1/2}}\right)}\\
		&= \frac{c_0}{2}  \int_{A}\abs{\nabla \tilde{H}\left(\frac{u_{\eps_k}}{\eps_k^{1/2}}\right)},
	\end{align*}
	where in the last line we have used the fact that $\tilde{H}(t)$ is constant for $\abs{t} \geq 1$. Now using the lower semicontinuity of the BV-norm with respect to $L^1$ convergence we have that 
	\begin{align*}
		\liminf_{k\to\infty} \int_{A\cap \{-\eps_k^{1/2} \leq u_{\eps_k} \leq \eps_k^{1/2}\}} \left(\abs{\nabla u_{\eps_k}}^2 + \frac{1}{\eps_k}W\left(\frac{u_{\eps_k}}{\eps_k^{1/2}}\right)\right)
		&\geq \frac{c_0}{2}  \int_A \abs{\nabla(\chi_{\Omega}-\chi_{\Omega^c})} \\
		&= c_0 Per\left( \Omega, A \right).
	\end{align*}
	Since $\int_{A\cap \{\abs{u_{\eps_k}} \geq \eps_k^{1/2}\}} \abs{\nabla u_{\eps_k}}^2 = \int_{A} \abs{\nabla (\abs{u_{\eps_k}}-\eps_k^{1/2})_{+}}^2$ and the Dirichlet energy is lower semi-continuous with respect to the $L^2$ convergence, this establishes the Lemma.
\end{proof}

\begin{lemma}
	\label{lemma:gamma:limsup}
	Given an admissible pair $(u,\Omega) \in \mathscr{A}(A)$, where $\overline{A}$ is a Lipschitz domain, and any $\eps_k\to0$, there exists a sequence $\{u_{\eps_k}\}_{k\in\N}\subset H^1(A)$ such that $u_{\eps_k} \to u$ in $L^2(A)$ and $\tilde{H}\left(\frac{u_{\eps_k}}{\eps_k^{1/2}}\right) \to \chi_{\Omega}-\chi_{\Omega^c}$ in $L^1(A)$ as $k \to \infty$. Moreover
	\begin{equation}
		\label{equation:limsup:eq}
		\limsup_{k\to\infty} E_{\eps_k}(u_{\eps_k},A) \leq E_A(u,\Omega).
	\end{equation}
\end{lemma}
\begin{proof}
	We will prove the Lemma under the additional assumption that $u$ is continuous on $\overline{A}$, the general case follows by approximation. We approximate $\Omega$ with a sequence of smooth sets $\Omega_j$ such that 
	\begin{equation}
		\label{eqn:convergence:perimeters}
		Per(\Omega_j,A) \to Per(\Omega,A) ,
	\end{equation}
	\begin{equation}
		\label{eqn:convergence:l1}
		\chi_{\Omega_j} \to \chi_{\Omega} \text{ in } L^1(A), 
	\end{equation}
	and
	\begin{equation}
		\label{eqn:convergence:difference}
		\abs{(\Omega_j \cap A) \Delta \Omega} \to 0,
	\end{equation}
	as $j \to \infty$, see for instance \cite{modica1987}. We will let $s_{j}(x)$ denote the signed distance to the set $\p\Omega_j \cap B_1$ with the convention that $s_j(x) >0$ if $x \in \Omega_j$.
	
	Fix a sequence $(\gamma_l)_{l\in\N}$ such that $\gamma_l \to 0$ as $ l \to \infty$. For each $l \in \N$ fix $j=j(l) \in \N$ large so that
	\begin{equation}
		\label{equation:perimeters:l}
		Per(\Omega_j,A) \leq Per(\Omega,A) + \gamma_{l},
	\end{equation}
	\begin{equation}
		\label{equation:l1:l}
		\norm{\chi_{\Omega_j} - \chi_{\Omega}}_{L^1(A)} \leq \gamma_{l},
	\end{equation}
	and
	\begin{equation}
		\label{equation:sup:j}
		\sup_{\p \Omega_j \cap A} \abs{u} \leq \gamma_l.
	\end{equation}
	Here \eqref{equation:perimeters:l} and \eqref{equation:l1:l} follow from \eqref{eqn:convergence:perimeters} and \eqref{eqn:convergence:l1} respectively, while \eqref{equation:sup:j} holds thanks to \eqref{eqn:convergence:difference} and the fact that $u=0$ on $ \p \Omega$ (since $u$ is continuous on $\overline{A}$). Now fix $\kappa = \kappa(l) > 0$ be small enough so that 
	\begin{equation}
		\label{equation:kP}
		\kappa Per(\Omega_j,A) \leq \gamma_l,
	\end{equation}
	and
	\begin{equation*}
		\sup_{ t_{\eps_k} \leq \abs{s} \leq \diam(A)} \frac{1}{\eps_k}W\left(\frac{\varphi_{\eps_k}(s)}{\eps_k^{1/2}}\right) \leq \frac{\gamma_l}{2\abs{A}},
	\end{equation*}
	where $\varphi_{\eps_k}$ and $t_{\eps_k}=t_{\eps_k}(\kappa)$ are defined in Lemma \ref{lemma:1d:properties}.
	 Defining the constants 
	\begin{equation*}
		\overline{\delta}_{k,j}=\sup_{0\leq s_j(x)\leq t_{\eps_k}} u(x), \hspace{2mm} \underline{\delta}_{k,j}=\inf_{-t_{\eps_k}\leq s_j(x)\leq0}u(x),
	\end{equation*}
	we let
	\begin{align*}
		u_{k,j}(x) &=\max\{\varphi_{\eps_k}(s_{j}(x)),(u-\overline{\delta}_{k,j})_+\}\chi_{\{s_j(x)\geq0\}} \\ &\hspace{3mm}+ \min\{\varphi_{\eps_k}(s_{j}(x)),-(u-\underline{\delta}_{k,j})_-\}\chi_{\{s_j(x)<0\}}.
	\end{align*}	
	Since $u$ is continuous on $\overline{A}$ the constants $\overline{\delta}_{k,j}$ and $\underline{\delta}_{k,j}$ are finite, and since $\p \Omega_j$ is smooth, we have that $u_{k,j}\in H^1(A)$ for each $k,j \in \N$. At this point we will fix an $l \in \N$, which in turn fixes $j=j(l)$ and $\kappa = \kappa(l)$ as above and we will find a suitable $k$ depending on $l$ such that conclusions of the Lemma hold. \\
	
	\textit{Step 1: $L^2$ convergence of $u_{k,j}$ to $u$.}\\
	Observe that $\lim_{k\to\infty} \overline{\delta}_{k,j} = \sup_{\p \Omega_j\cap \overline{A}} u =:\overline{\delta}_j$ and $\lim_{k\to\infty} \underline{\delta}_{k,j} = \inf_{\p \Omega_j\cap \overline{A}} u=:\underline{\delta}_j$.  Hence, using \eqref{equation:sup:j} and the fact that $\varphi_{\eps_k} \to 0$ uniformly as $k \to \infty$, we have for $k \geq k_0(l)$ that
	\begin{equation}
		\label{equation:intermediate:l2}
		\norm{u_{k,j} - u}_{L^2(A)} \leq C\gamma_l.
	\end{equation}
	\\
	\textit{Step 2: Establishing \eqref{equation:limsup:eq}.}\\
	Notice that if $\abs{u_{k,j}} \leq (1-\kappa)\eps_k^{1/2}$ we have that $u_{k,j}(x)=\varphi_{\eps_k}(s_j(x))$ so that by \eqref{equation:1d:property} there holds
	\begin{align*}
		\abs{\nabla u_{k,j}} = \abs{\varphi_{\eps_k}'(s_{j}(x))}\abs{\nabla s_{j}(x)} = \frac{1}{\eps_k^{1/2}}\sqrt{W\left(\frac{\varphi_{\eps_k}(s_{j}(x))}{\eps_k^{1/2}}\right)} \text{ on } A\cap\{\abs{u_{k,j}} \leq (1-\kappa)\eps_k^{1/2}\}.
	\end{align*}
	Therefore, using the coarea formula, we find that
	\begin{align*}
		\int_{A\cap\{\abs{u_{k,j}} \leq (1-\kappa)\eps_k^{1/2}\}} \left(\abs{\nabla u_{k,j}}^2 + \frac1{\eps_k}W\left(\frac{u_{k,j}}{\eps_k^{1/2}}\right)\right)
		&= \int_{A\cap\{\abs{u_{k,j}} \leq (1-\kappa)\eps_k^{1/2}\}} \frac{2}{\eps_k^{1/2}}\sqrt{W\left(\frac{\varphi_{\eps_k}(s_{j}(x))}{\eps_k^{1/2}}\right)}\abs{\nabla \varphi_{\eps_k}(s_{j}(x))}\\
		&= \int_{A\cap\{\abs{u_{k,j}} \leq (1-\kappa)\eps_k^{1/2}\}} \abs{\nabla H\left(\frac{\varphi_{\eps_k}(s_j(x))}{\eps_k^{1/2}}\right)}\\
		&= \int_{-1+\kappa}^{1-\kappa} 2\sqrt{W(s)} Per\left(\left\{\frac{\varphi_{\eps_k}(s_{j}(x))}{\eps_k^{1/2}}>s\right\}, A\right).
	\end{align*}
	Since $u_{k,j} \geq \varphi_{\eps_k}(s_j(x))$ on $\{s_j \geq 0\}$ and $u_{k,j} \leq \varphi_{\eps_k}(s_j(x))$ on $\{s_j \leq 0\}$, there holds that $W(\frac{u_{k,j}}{\eps_k^{1/2}})\leq W(\frac{ \varphi_{\eps_k}(s_j(x))}{\eps_k^{1/2}})$. Therefore we have that
	\begin{align*}
		\int_{A} \left(\abs{\nabla u_{k,j}}^2 + \frac1{\eps_k}W\left(\frac{u_{k,j}}{\eps_k^{1/2}}\right)\right)
		&=\int_{-1+\kappa}^{1-\kappa} 2\sqrt{W(s)} Per\left(\left\{\frac{\varphi_{\eps_k}(s_{j}(x))}{\eps_k^{1/2}}>s\right\}, A\right)\\
		&\hspace{2mm}+\int_{A\cap\{\abs{u_{k,j}} \geq (1-\kappa) \eps_k^{1/2}\}} \abs{\nabla u_{k,j}}^2 + \int_{A\cap\{\abs{u_{k,j}} \geq (1-\kappa)\eps_k^{1/2}\}} \frac{1}{\eps_k}W\left(\frac{u_{k,j}}{\eps_k^{1/2}}\right) \\
		&\leq c_0Per(\Omega_j,A) + C\kappa Per(\Omega_j,A)+ \int_{A\cap\{\abs{u_{k,j}} \geq (1-\kappa) \eps_k^{1/2}\}} \abs{\nabla u_{k,j}}^2 \\
		&\hspace{2mm}+ \int_{A\cap\{t_{\eps_k} \leq s_j(x) \leq \diam(A)\}} \frac{1}{\eps_k}W\left(\frac{\varphi_{\eps_{k}}(s_j(x))}{\eps_k^{1/2}}\right)  + o(1),
	\end{align*}
	as $k \to\infty $. 
	Using \eqref{equation:kP}, the fact that $\nabla (u-\overline{\delta}_{k,j})_+ \to \nabla (u-\overline{\delta}_j)_+$ and $\nabla (u-\underline{\delta}_{k,j})_- \to \nabla (u-\underline{\delta}_j)_-$ as $k \to \infty$ in $L^2$, as well as the estimate 
	\begin{equation*}
		\int_{A\cap\{t_{\eps_k} \leq s_j(x) \leq \diam(A)\}} \frac{1}{\eps_k}W\left(\frac{\varphi_{\eps_{k}}(s_j(x))}{\eps_k^{1/2}}\right) \leq \sup_{ t_{\eps_{k}} \leq \abs{s} \leq \diam(A) } \frac{1}{\eps_{k}}W\left(\frac{\varphi_{\eps_k}(s)}{\eps_k^{1/2}}\right) \abs{A} \leq \frac{\gamma_l}{2},
	\end{equation*}
	we find that
	\begin{align*}
		\int_{A} \left(\abs{\nabla u_{k,j}}^2 + \frac1{\eps_k}W\left(\frac{u_{k,j}}{\eps_k^{1/2}}\right)\right)
		&\leq  c_0Per(\Omega_j,A) + \int_{A\backslash\{\underline{\delta}_j \leq u \leq \overline{\delta}_j\}} \abs{\nabla u}^2 + C\gamma_{l}+ o(1),
	\end{align*}
	as $k \to \infty$. Therefore for $k \geq k_1(l)\geq k_0(l)$ we have 
	\begin{align}
		\label{equation:intermediate:bound}
		\int_{A} \left(\abs{\nabla u_{k,j}}^2 + \frac1{\eps_k}W\left(\frac{u_{k,j}}{\eps_k^{1/2}}\right)\right)&\leq  c_0Per(\Omega,A) + \int_{A} \abs{\nabla u}^2 + C\gamma_l.
	\end{align}
	
	\textit{Step 3: $L^1$ convergence of $\tilde{H}\left(\frac{u_{k,j}}{\eps_k^{1/2}}\right)$ to $\chi_{\Omega}-\chi_{\Omega^c}$.}\\
	Observe by Young's inequality that
	\begin{align*}
		\int_{A} \abs{\nabla \tilde{H}\left(\frac{u_{k,j}}{\eps_k^{1/2}}\right)} =\frac{2}{c_0} \int_{A} 2\sqrt{W\left(\frac{u_{k,j}}{\eps_k^{1/2}}\right)} \abs{\nabla u_{k,j}}\eps_k^{-1/2} \leq \frac{2}{c_0} E_{\eps_k}(u_{k,j},A),
	\end{align*}
	so that by \eqref{equation:intermediate:bound} and the fact that $\abs{\tilde{H}}\leq 1$ we have that $\tilde{H}\left(\frac{u_{k,j}}{\eps_k^{1/2}}\right)$ is bounded in $BV(A)$ for $k\geq k_1(l)$. Consequently, up to taking a subsequence in $k$, $\tilde{H}\left(\frac{u_{k,j}}{\eps_k^{1/2}}\right) \to g_j$ in $L^1(A)$ as $k \to \infty$ for some $g_j \in L^1(A)$. Now we claim that 
	\begin{equation}
		\label{equation:gj:sets}
		g_j = \chi_{\Omega_j} - \chi_{\Omega_j^c}
	\end{equation}
	Indeed, let $\delta > 0$ small and notice that
	\begin{align*}
		\abs{\left\{-1+2\delta \leq g_j \leq 1-2\delta\right\}\cap A}
		&\leq \abs{\left\{-1+\delta \leq 	\tilde{H}\left(\frac{u_{k,j}}{\eps_k^{1/2}}\right) \leq 1-\delta\right\}\cap A} + o(1)\\
		&\leq \abs{\left\{\tilde{H}^{-1}\left(-1+\delta\right) \leq  \frac{u_{k,j}}{\eps_k^{1/2}}\leq \tilde{H}^{-1}\left(1-\delta\right)\right\}\cap A}+o(1)\\
		&\leq C(\delta,W) \int_{A} W\left(\frac{u_{k,j}}{\eps_k^{1/2}}\right) +o(1),
	\end{align*}
	as $k \to \infty$. Here we have used the fact that for $\delta$ small we have $C(\delta, W) W(t)\geq 1$ on the set $\left\{\tilde{H}^{-1}(-1+\delta) \leq  t \leq \tilde{H}^{-1}(1-\delta)\right\}$, where $C(\delta,W)$ is a large constant depending only on $\delta$ and $W$. Now using \eqref{equation:intermediate:bound} once more we have that
	\begin{equation*}
		\int_{A}\frac{1}{\eps_k}W\left(\frac{u_{k,j}}{\eps_k^{1/2}}\right) \leq C,
	\end{equation*}
	and so
	\begin{align*}
		\abs{\left\{-1+2\delta \leq g \leq 1-2\delta\right\}\cap B_R} \leq C\eps_k + o(1)\to 0 \text{	as $k \to \infty$,}
	\end{align*}
	which establishes that $g_j \in \left\{1,-1\right\}$. This however establishes \eqref{equation:gj:sets} since	
	\begin{equation*}
		\tilde{H}\left(\frac{u_{k,j}}{\eps_k^{1/2}}\right) \geq 0 \text{ on } \Omega_j,
	\end{equation*}
	and 
	\begin{equation*}
		\tilde{H}\left(\frac{u_{k,j}}{\eps_k^{1/2}}\right) \leq 0 \text{ on } \Omega_j^c.
	\end{equation*}
	Therefore, taking $k\geq k_2(l)\geq k_1(l)$ we have that
	\begin{equation}
		\label{equation:intermediate:H}
		\norm{\tilde{H}\left(\frac{u_{k,j}}{\eps_k^{1/2}}\right)-\chi_{\Omega}+\chi_{\Omega^c}}_{L^1(A)} \leq C\gamma_l.
	\end{equation}
	
	\textit{Step 4: Constructing the sequence.}
	We will now construct a sequence $u_{\eps_{k_l}}$ satisfying the conclusions of the Lemma. For each $l \in \N$ we take $j_l := j(l)$ and $k_l \geq k_2(l)$ and we set $u_{\eps_{k_l}} = u_{k_l,j_l}$. Notice that \eqref{equation:intermediate:l2}, \eqref{equation:intermediate:bound} and \eqref{equation:intermediate:H} hold for $u_{\eps_{k_l}}$. From these we conclude that $u_{\eps_{k_l}} \to u$ in $L^2(A)$ and $\tilde{H}\left(\frac{u_{\eps_{k_l}}}{\eps_{k_l}^{1/2}}\right) \to \chi_{\Omega}-\chi_{\Omega^c}$ in $L^1(A)$ as $l \to \infty$, and  \eqref{equation:limsup:eq} holds.
\end{proof}

\section{Proof of Theorem \ref{theorem:main:existence}}\label{section:prooffinal}
We begin with the following interpolation result which is an adaptation of \cite[Proposition 5.1]{savinAltPhillips} for our problem.
\begin{proposition}
	\label{proposition:interpolation}
	Let $u_{\eps_k}$ and $v_{\eps_k}$ be sequences in $\tilde{\mathscr{A}}_M(B_1)$  with $\eps_{k}\to 0$. Assume that for some $\rho \in (1/2,1)$ and each $\delta > 0$ small that
	\begin{equation}
		\label{eqn:energybounds}
		E_{\eps_k}(u_{\eps_k}, B_{\rho+\delta}) + E_{\eps_k}(v_{\eps_k},B_{\rho+\delta}) \leq C,
	\end{equation}
	uniformly in $k$, and that
	\begin{equation}
		\label{equation:convergence}
		\norm{u_{\eps_k}-v_{\eps_k}}_{L^2(B_1 \backslash \overline{B}_{\rho})}  \to 0,
	\end{equation}
	and
	\begin{equation}
		\label{equation:convergence:H}
		\norm{\tilde{H}\left(\frac{u_{\eps_k}}{\eps_k^{1/2}}\right)-\tilde{H}\left(\frac{v_{\eps_k}}{\eps_k^{1/2}}\right)}_{L^1(B_{1}\backslash \overline{B}_{\rho})}  \to 0,
	\end{equation}
	as $k\to\infty$.
	Then for any $\gamma >0$ there exists a $\delta > 0$ and a sequence $w_{\eps_k} \in \tilde{\mathscr{A}}_{M}(B_1)$ such that
	\begin{equation*}
		w_{\eps_k} =
		\begin{cases}
			v_{\eps_k} & \text{ in } B_{\rho}, \\
			u_{\eps_k} & \text{in } B_{1} \backslash \overline{B}_{\rho+\delta},
		\end{cases}
	\end{equation*}
	and for $k$ large enough there holds that
	\begin{equation}
	\label{equation:interpolation:energy}
	E_{\eps_k}(w_{\eps_k},B_1) \leq E_{\eps_k}(v_{\eps_k},B_{\rho}) + E_{\eps_k}(u_{\eps_k},B_1\backslash \overline{B}_{\rho+\delta}) + \gamma.
	\end{equation}
\end{proposition}
\begin{proof}
	 Let $\gamma > 0$ and notice that by \eqref{eqn:energybounds} we can pick $\delta$ small enough such that
	\begin{equation}
		\label{equation:energy:gamma}
		E_{\eps_k}(u_{\eps_k},B_{\rho+\delta}\backslash B_{\rho}) \leq \frac{\gamma}{3} \text{ and }	E_{\eps_k}(v_{\eps_k},B_{\rho+\delta}\backslash B_{\rho}) \leq \frac{\gamma}{3}. 
	\end{equation}
	Moreover, we only need to prove the conclusion under the assumption that $u_{\eps_k} \geq v_{\eps_k}$. Indeed, the interpolation follows for unordered sequences by first interpolating between $u_{\eps_k}$ and $\min\{u_{\eps_k},v_{\eps_k}\}$ in $B_{\rho+\delta} \backslash B_{\rho+\delta/2}$, and then between $\min\{u_{\eps_k},v_{\eps_k}\}$  and $v_{\eps_k}$ in the annulus $B_{\rho+\delta/2}\backslash B_{\rho}$.
	Thus assuming $u_{\eps_k} \geq v_{\eps_k}$ in $B_1$, we define the constant $\theta_{\delta} = 16M\delta^{-1}$, and for each $r \in \left[\rho+\delta/8,\rho+\delta/4\right]$ we set
	\begin{equation*}
		\psi_{r,k} (x) = \varphi_{\theta_{\delta},\eps_k} (\abs{x}-r),
	\end{equation*}
	where $\varphi_{\theta_{\delta},\eps_k}$ is defined in Lemma \ref{lemma:1D:properties:theta}. We notice that with this choice of $\theta_{\delta}$, we have that $\psi_{r,k}\leq -M$ on $B_{\rho}$ and $\psi_{r,k} \geq M$ on $B_{1} \backslash B_{\rho+\delta}$, as long as $k$ is chosen sufficiently large. Therefore, defining
	\begin{equation}
		\Psi_{r,k} = \min\{u_{\eps_k},\max\{v_{\eps_k},\psi_{r,k}\}\},
	\end{equation}
	we have that $\Psi_{r,k}=v_{\eps_k}$ in $B_{\rho}$ while $\Psi_{r,k} = u_{\eps_k}$ in $B_{1} \backslash B_{\rho+\delta}$. Using now \eqref{equation:energy:gamma}, we have that 
	\begin{align*}
		E_{\eps_k}(\Psi_{r,k},B_1)
		&\leq E_{\eps_k}(v_{\eps_k},B_{\rho}) + E_{\eps_k}(u_{\eps_k},B_1 \backslash B_{\rho+\delta}) + E_{\eps_k}(\Psi_{r,k},B_{\rho+\delta}\backslash B_{\rho})\\
		&\leq E_{\eps_k}(v_{\eps_k},B_{\rho}) + E_{\eps_k}(u_{\eps_k},B_1 \backslash B_{\rho+\delta}) + E_{\eps_k}(u_{\eps_k},B_{\rho+\delta}\backslash B_{\rho}) + E_{\eps_k}(v_{\eps_k},B_{\rho+\delta}\backslash B_{\rho}) \\& \hspace{2mm}+E_{\eps_k}(\psi_{r,k},B_{\rho+\delta}\backslash B_{\rho} \cap \{v_{\eps_k}<\Psi_{r,k}<u_{\eps_k}\})\\
		&\leq E_{\eps_k}(v_{\eps_k},B_{\rho}) + E_{\eps_k}(u_{\eps_k},B_1 \backslash B_{\rho+\delta})\\&\hspace{2mm} + E_{\eps_k}(\psi_{r,k},(B_{\rho+\delta}\backslash B_{\rho}) \cap \{v_{\eps_k}<\Psi_{r,k}<u_{\eps_k}\}) + \frac23\gamma.
	\end{align*}
	Now setting $D_{r,k} := (B_{\rho+\delta}\backslash B_{\rho}) \cap \{v_{\eps_k}<\Psi_{r,k}<u_{\eps_k}\}$, we have using \eqref{equation:1d:theta}, the coarea formula, and the fact that $-M \leq v_{\eps_k} < \psi_{r,k} < u_{\eps_k} \leq M$ on $D_{r,k}$  that
	\begin{align*}
		E_{\eps_k} (\psi_{r,k},D_{r,k})
		&\leq \int_{D_{r,k}} 2\abs{\nabla \psi_{r,k}}^2 \\
		&= \int_{D_{r,k}} 2 \abs{\nabla \psi_{r,k}}\sqrt{\frac{1}{\eps_k}W\left(\frac{\psi_{r,k}}{\eps_k^{1/2}}\right) + \theta_{\delta}}\\
		&\leq \int_{D_{r,k}} \frac{2}{\eps_k^{1/2}}\abs{\nabla \psi_{r,k}}\sqrt{W\left(\frac{\psi_{r,k}}{\eps_k^{1/2}}\right)} + \int_{D_{r,k}}2 \abs{\nabla \psi_{r,k}} \sqrt{\theta_{\delta}}\\
		&= \frac{c_0}{2}\int_{D_{r,k}} \abs{\nabla \tilde{H}\left(\frac{\psi_{r,k}}{\eps_k^{1/2}}\right)} +  \int_{D_{r,k}}2 \abs{\nabla \psi_{r,k}} \sqrt{\theta_{\delta}}\\
		&= \int_{-1}^{1} \mathcal{H}^{n-1}\left(\left\{ \tilde{H}\left(\frac{\psi_{r,k}}{\eps_k^{1/2}}\right) = s\right\} \cap D_{r,k}\right) ds\\ &\hspace{2mm}+ 2\sqrt{\theta_{\delta}}\int_{-M}^{M} \mathcal{H}^{n-1}\left(\left\{ \psi_{r,k} = s\right\} \cap D_{r,k}\right) ds\\
		&= \int_{-1}^{1} \mathcal{H}^{n-1}\left(\left\{ \frac{\psi_{r,k}}{\eps_k^{1/2}} = s\right\} \cap D_{r,k}\right) 2\sqrt{W(s)} ds\\ &\hspace{2mm}+ 2\sqrt{\theta_{\delta}}\int_{-M}^{M} \mathcal{H}^{n-1}\left(\left\{ \psi_{r,k} = s\right\} \cap D_{r,k}\right) ds.
	\end{align*}
	In the last line we have used the change of variables $s=\tilde{H}(t)$ in the first integral. Now using the fact that
	\begin{equation*}
		\left\{\frac{\psi_{r,k}}{\eps_{k}^{1/2}} = s\right\} \cap D_{r,k} = \left\{\frac{v_{\eps_k}}{\eps_k^{1/2}} < s < \frac{u_{\eps_k}}{\eps_k^{1/2}}\right\} \cap \p B_{r+\varphi_{\theta_{\delta},k}^{-1}(\eps_k^{1/2}s)} \cap (B_{\rho+\delta}\backslash B_{\rho}),
	\end{equation*}
	and
	\begin{equation*}
		\left\{\psi_{r,k}= s\right\} \cap D_{r,k} = \left\{v_{\eps_k} < s < u_{\eps_k}\right\} \cap \p B_{r+\varphi_{\theta_{\delta},k}^{-1}(s)} \cap (B_{\rho+\delta}\backslash B_{\rho}),
	\end{equation*}
	we can average in $r \in [\rho+\frac{\delta}{8}, \rho+\frac{\delta}{4} ]$ and obtain
	\begin{align*}
		\strokedint_{\rho+\delta/8}^{\rho+\delta/4} E_{\eps_k} (\psi_{r,k},D_{r,k})dr
		&\leq \frac{C}{\delta}\int_{-1}^{1} \mathcal{H}^{n}\left(\left\{\frac{v_{\eps_k}}{\eps_k^{1/2}} < s < \frac{u_{\eps_k}}{\eps_k^{1/2}}\right\} \cap  (B_{\rho+\delta}\backslash B_{\rho})\right) 2\sqrt{W(s)} ds \\&\hspace{2mm}+  \frac{C}{\delta^{3/2}}\int_{-M}^{M} \mathcal{H}^{n}\left( \left\{v_{\eps_k} < s < u_{\eps_k}\right\} \cap (B_{\rho+\delta}\backslash B_{\rho})\right) ds.
	\end{align*}
	For the second integral on the right, we have by applying Fubini's theorem that
	\begin{align*}
		\int_{-M}^{M} \mathcal{H}^{n}\left( \left\{v_{\eps_k} < s < u_{\eps_k}\right\} \cap (B_{\rho+\delta}\backslash B_{\rho})\right) ds
		&= 	\int_{-M}^{M} \int_{B_{\rho+\delta}\backslash B_{\rho}} \chi_{\left\{v_{\eps_k} < s < u_{\eps_k}\right\}} dxds\\
		&=  \int_{B_{\rho+\delta}\backslash B_{\rho}} 	\int_{-M}^{M} \chi_{\left\{v_{\eps_k} < s < u_{\eps_k}\right\}} dsdx\\
		&= \int_{B_{\rho+\delta}\backslash B_{\rho}} \abs{u_{\eps_k}-v_{\eps_k}}.
	\end{align*}
	Now making the change of variables $t = \tilde{H}(s)$ in the first integral and then arguing as above we have that
	\begin{align*}
		\int_{-1}^{1} \mathcal{H}^{n}\left(\left\{\frac{v_{\eps_k}}{\eps_k^{1/2}} < s < \frac{u_{\eps_k}}{\eps_k^{1/2}}\right\} \cap  (B_{\rho+\delta}\backslash B_{\rho})\right) 2\sqrt{W(s)} ds
		&= \int_{-1}^{1} \mathcal{H}^{n}\left(\left\{ \tilde{H}\left(\frac{v_{\eps_k}}{\eps_k^{1/2}}\right) < s < \tilde{H}\left(\frac{u_{\eps_k}}{\eps_k^{1/2}}\right)\right\} \cap  (B_{\rho+\delta}\backslash B_{\rho})\right) ds\\
		&= \int_{B_{\rho+\delta}\backslash B_{\rho}} \abs{\tilde{H}\left(\frac{u_{\eps_k}}{\eps_k^{1/2}}\right)-\tilde{H}\left(\frac{v_{\eps_k}}{\eps_k^{1/2}}\right)}.
	\end{align*}
	Thus, we obtain
	\begin{align*}
		\strokedint_{\rho+\delta/8}^{\rho+\delta/4} E_{\eps_k} (\psi_{r,k},D_{r,k})dr
		\leq \frac{C}{\delta}\left( \norm{\tilde{H}\left(\frac{u_{\eps_k}}{\eps_k^{1/2}}\right)-\tilde{H}\left(\frac{v_{\eps_k}}{\eps_k^{1/2}}\right)}_{L^1(B_{\rho+\delta}\backslash B_{\rho})} +  \norm{u_{\eps_k}-v_{\eps_k}}_{L^2(B_{\rho+\delta}\backslash B_{\rho})}\right) \leq \gamma/3,
	\end{align*}
	for $k$ sufficiently large thanks to \eqref{equation:convergence} and \eqref{equation:convergence:H}. Consequently, for each such $k$ there is an $r_k \in [\rho+\delta/8,\rho+\delta/4]$ such that
	\begin{equation*}
		E_{\eps_k}(\psi_{r_k,k}, D_{r_k,k}) \leq \frac{\gamma}{3}.
	\end{equation*}
	Therefore, setting $w_{\eps_k} = \Psi_{r_k,k}$ yields the desired interpolation. 
\end{proof}

We also have the following (see \cite[Lemma A.4]{fig04}).
\begin{proposition}
	\label{proposition:boundedness}
	Let $u_{\eps_k}$ be a local minimiser for $E_{\eps_k}(\cdot, B_1)$ over the set $\tilde{\mathscr{A}}_M(B_1)$. Then for each $ 0<R<1$ we have that
	\begin{equation}
		\label{equation:energy:bounds}
		E_{\eps_k}(u_{\eps_k},B_R) \leq C(R,M)
	\end{equation}
	and 
	\begin{equation}
		\label{equation:h1:bound}
	\norm{u_{\eps_k}}_{H^1(B_R)} \leq C(R,M)
	\end{equation}
	for $k \in \N$ large enough. 
\end{proposition}
\begin{proof}
	Let 
	\begin{equation*}
		\tilde{v}=
		\begin{cases}
			M \min\{\frac{2}{1-R}\dist(x, \p B_{\frac{R+1}{2}}),1\} & x \in B_{\frac{R+1}{2}} \\
			M\max\{-\frac{2}{1-R}\dist(x, \p B_{\frac{R+1}{2}}),-1\} & x \in B^c_{\frac{R+1}{2}} 
		\end{cases}
	\end{equation*}
	and notice that $\tilde{v} = M$ on $B_{R}$.  Letting $s_R$ be the signed distance to $\p B_R$ with the convention that $s_R > 0$ on $B_R$ we define the sequence of functions
	\begin{align*}
		v_{\eps_k}(x) = &\max\{\varphi_{\eps_k}(s_{R}(x)),(\tilde{v}-\frac{2M}{1-R}t_{\eps_k})_+\}\chi_{\{s_R(x)\geq0\}} \\
		&+ \min\{\varphi_{\eps_k}(s_{R}(x)),-(\tilde{v}+ \frac{2M}{1-R}t_{\eps_k})_-\}\chi_{\{s_R(x)<0\}},
	\end{align*}
	where $\varphi_{\eps_k}$ and $t_{\eps_k}$ are defined in Lemma \ref{lemma:1d:properties}, so that arguing as in the proof Lemma \ref{lemma:gamma:limsup}, we have for $k$ large enough that
	\begin{align*}
		E_{\eps_k}(v_{\eps_k},B_1)
		&\leq \int_{-1}^1 2\sqrt{W(s)} Per\left(\left\{\frac{\varphi_{\eps_k}(s_{R}(x))}{\eps_k^{1/2}}>s\right\}, B_1\right) +\int_{B_1\cap\{\abs{v_{\eps_k}}\geq \eps_k^{1/2}\}} \abs{\nabla v_{\eps_k}}^2 + \frac{1}{2}\\
		&\leq c_0 \mathcal{H}^{n-1}(\p B_R) + \int_{B_1} \abs{\nabla \tilde{v}}^2 + 1.
	\end{align*}
	Since $\abs{\nabla \tilde{v}} \leq C(R,M)$ we have that
	\begin{equation}
		\label{equation:veps:bounded}
			E_{\eps_k}(v_{\eps_k},B_1) \leq C(R,M).
	\end{equation}
	Since $u_{\eps_k} \leq v_{\eps_k}$ on $B_R$ for $k$ large enough, and $u_{\eps_k} \geq v_{\eps_k}$ on $\p B_1$, we have that 
	\begin{align*}
		E_{\eps_k}(u_{\eps_k},B_R) \leq E_{\eps_k}(u_{\eps_k},\{u_{\eps_k} > v_{\eps_k}\}) \leq E_{\eps_k}(v_{\eps_k},\{u_{\eps_k} > v_{\eps_k}\}) \leq E_{\eps_k}(v_{\eps_k},B_1),
	\end{align*}
	which with \eqref{equation:veps:bounded} establishes \eqref{equation:energy:bounds}. 
	Since $\abs{u_{\eps_k}} \leq M$ we immediately obtain \eqref{equation:h1:bound} which concludes the proof. 
\end{proof}

We can now give the
\begin{proof}[Proof of Theorem \ref{theorem:main:existence}]
	We observe that thanks to Proposition \ref{proposition:boundedness} we have up to a subsequence that
	\begin{equation*}
		u_{\eps_k} \to u \text{ in }L^2_{\loc}(B_1) \text{ as } k \to \infty.
	\end{equation*}
	Moreover, by Young's inequality, we have for each $R>0$ that
	\begin{align*}
		\int_{B_R} \abs{\nabla \tilde{H}\left(\frac{u_{\eps_k}}{\eps_k^{1/2}}\right)} =\frac{2}{c_0} \int_{B_R} 2\sqrt{W\left(\frac{u_{\eps_k}}{\eps_k^{1/2}}\right)} \abs{\nabla u_{\eps_k}}\eps_k^{-1/2} \leq \frac{2}{c_0} E_{\eps_k}(u_{\eps_k},B_R).
	\end{align*}
	Hence, using Proposition \ref{proposition:boundedness} and the fact that $\abs{\tilde{H}}\leq 1$, we conclude that $\tilde{H}\left(\frac{u_{\eps_k}}{\eps_k^{1/2}}\right)$ is bounded in $BV(B_R)$. Therefore, up to taking another subsequence, we have that 
	\begin{equation*}
		 \tilde{H}\left(\frac{u_{\eps_k}}{\eps_k^{1/2}}\right) \to g \text{ in } L^1_{\loc}(B_1) \text{ as } k \to \infty,
	\end{equation*}
	for some $g \in L^1_{\loc}(B_1)$. Now we claim that 
	\begin{equation}
		\label{equation:g:sets}
		g = \chi_{\Omega} - \chi_{\Omega^c}
	\end{equation}
	for some measurable set $\Omega$. Indeed, let $\delta > 0$ small and arguing as in the proof of Lemma \ref{lemma:gamma:limsup} we find that
	\begin{align*}
		\abs{\left\{-1+2\delta \leq g \leq 1-2\delta\right\}\cap B_R}
		&\leq C(\delta,W) \int_{B_R} W\left(\frac{u_{\eps_k}}{\eps_k^{1/2}}\right) +o(1),
	\end{align*}
	 as $k \to \infty$, so that by Proposition \ref{proposition:boundedness} 
	\begin{align*}
		\abs{\left\{-1+2\delta \leq g \leq 1-2\delta\right\}\cap B_R} \to 0 \text{	as $k \to \infty$,}
	\end{align*}
	 which establishes \eqref{equation:g:sets}. Now since 
	\begin{equation*}
		\tilde{H}\left(\frac{u_{\eps_k}}{\eps_k^{1/2}}\right) =-1 \text{ for } u_{\eps_k} \leq -\eps_k^{1/2},
	\end{equation*}
	and 
	\begin{equation*}
		\tilde{H}\left(\frac{u_{\eps_k}}{\eps_k^{1/2}}\right) =1 \text{ for } u_{\eps_k} \geq \eps_k^{1/2},
	\end{equation*}
	we have that $u \geq 0$ a.e. on $\Omega$ while $u\leq 0$ a.e. on $\Omega^c$. Moreover, since $\abs{u_{\eps_k}}\leq M$ we have that $\abs{u} \leq M$ and so $(u,\Omega) \in \mathscr{A}_M(B_1)$ (the fact that $\Omega$ is a set of finite perimeter follows from \eqref{equation:final:aim} below). 
	
	Now let $(v,\Sigma) \in \mathscr{A}(B_1)$ such that $v=u$ on $B_1\backslash B_{R}$ for some $R<1$. We we will show that
	\begin{equation}
		\label{equation:final:aim}
		E_{B_1}(u,\Omega) \leq E_{B_1}(v,\Sigma).
	\end{equation}
	 To this end we let $v_{\eps_k}$ be the sequence of functions constructed in Lemma \ref{lemma:gamma:limsup} for the pair $(v,\Sigma)$, and since we can assume that $E_{B_{1}}(v,\Sigma) <\infty$ (or else \eqref{equation:final:aim} holds trivially) we can apply Proposition \ref{proposition:interpolation} to the sequences $u_{\eps_k}$ and $v_{\eps_k}$ with $\rho=R$. Hence for any $\gamma > 0$ there exists a $\delta > 0$ small and a sequence $w_{\eps_k}$ such that $w_{\eps_k} = u_{\eps_k}$ in $B_1 \backslash \overline{B}_{\rho+\delta}$. Moreover, for $k$ large enough we have by the minimality of $u_{\eps_k}$ and \eqref{equation:interpolation:energy} that
	\begin{align*}
		E_{\eps_k}(u_{\eps_k}, B_1) \leq E_{\eps_k}(w_{\eps_k},B_1) \leq E_{\eps_k}(v_{\eps_k},B_{\rho})+ E_{\eps_k}(u_{\eps_k},B_1\backslash \overline{B}_{\rho+\delta}) + \gamma.
	\end{align*}
	We therefore find using Theorem \ref{theorem:gamma:convergence} that
	\begin{align*}
		E_{B_{\rho+\delta}}(u,\Omega) 
		 &\leq\liminf_{k\to\infty} E_{\eps_k}(u_{\eps_k}, B_{\rho+\delta})\\
		 &\leq \limsup_{k\to\infty} E_{\eps_k}(v_{\eps_k},B_{\rho+\delta})+\gamma \\
		 &\leq E_{B_{\rho+\delta}}(v,\Sigma) + \gamma.
	\end{align*}
	Since $E_{B_1\backslash B_{\rho+\delta}}(u,\Omega) = E_{B_1\backslash B_{\rho+\delta}}(v,\Sigma)$, this establishes \eqref{equation:final:aim} and completes the proof. 
	\end{proof}

\bibliographystyle{alpha}
\bibliography{References.bib}
\end{document}